\documentclass[a4paper,12pt]{amsart}

\usepackage{amssymb, amsmath, amsthm, mathrsfs, braket, xspace}
\usepackage[margin=1.0in]{geometry}
\numberwithin{equation}{section}
\usepackage[colorlinks=true]{hyperref}
\usepackage{aliascnt}
\usepackage{mathtools}
\mathtoolsset{showonlyrefs=true}
\usepackage[all,2cell]{xy}
\objectmargin+{1mm}
\labelmargin+{0.8mm}
\SelectTips{cm}{12}
\usepackage{enumitem}
\setlist[enumerate]{itemsep=0pt,label=$({\rm \roman*})$, topsep=5pt}
\setlist[itemize]{itemsep=0pt, topsep=5pt, labelindent=\parindent, leftmargin=*}
\setlist[description]{itemsep=0pt, topsep=5pt, leftmargin=*}

\usepackage[T1]{fontenc}

\newtheorem{thm}{Theorem}[section]

\newaliascnt{cor}{thm}
\newtheorem{cor}[cor]{Corollary}
\aliascntresetthe{cor}

\newaliascnt{lem}{thm}
\newtheorem{lem}[lem]{Lemma}
\aliascntresetthe{lem}

\newaliascnt{prop}{thm}
\newtheorem{prop}[prop]{Proposition}
\aliascntresetthe{prop}

\newaliascnt{conj}{thm}

\aliascntresetthe{conj}

\theoremstyle{definition}
\newaliascnt{dfn}{thm}
\newtheorem{dfn}[dfn]{Definition}
\aliascntresetthe{dfn}

\newaliascnt{rem}{thm}

\aliascntresetthe{rem}

\newaliascnt{ex}{thm}
\newtheorem{ex}[ex]{Example}
\aliascntresetthe{ex}

\newcommand{\Bands}{\operatorname{\textbf{Bands}}}

\newcommand{\Null}{\mathcal{N}}

\title{Krasner Completions of Filtered Bands}

\author[T. Hiranouchi]{Toshiro Hiranouchi}
\address[T. Hiranouchi]{
Department of Basic Sciences, Graduate School of Engineering,
Kyushu Institute of Technology,
1-1 Sensui-cho, Tobata-ku, Kitakyushu-shi,
Fukuoka 804-8550 JAPAN}
\email{hira@mns.kyutech.ac.jp}


\begin{document}
\date{\today}
\begin{abstract}
We introduce the Krasner completion of a band equipped with a decreasing
filtration by subgroups of its unit group and give a criterion for
reconstruction from its quotient bands. We recover the Krasner--Linzi 
inverse-limit theorem for valued fields. We also obtain a valuation ring
analogue of the Krasner--Linzi theorem and compare Krasner completions of unit
bands with adic completion.
\end{abstract}
\maketitle

\section{Introduction}

Let $(K,v)$ be a valued field with value group $\Gamma$, let
$(\widehat K,\widehat v)$ be its completion, and, for
$\gamma\in\Gamma_{\geq0}$, put
\[
U_v^\gamma
=
\set{u\in K^\times | v(u-1)>\gamma},
\qquad
K_\gamma
=
K^\times/U_v^\gamma\cup\set{0}.
\]
In \cite{Kra57} Krasner proved that a complete valued field can be recovered from these
quotient hyperfields $K_\gamma$ in rank one.  Linzi \cite[Theorem~2]{Lin23} extended this result
to valuations of arbitrary Archimedean rank and gave a categorical
formulation.  In particular, there is a canonical
isomorphism
\begin{equation}\label{eq:kras}
\widehat K
\xrightarrow{\simeq}
\varprojlim_{\gamma\in\Gamma_{\geq0}}K_\gamma.
\end{equation}
Thus ordinary addition on the completed field is recovered from
finite-level quotients whose addition is multivalued.

Motivated by this reconstruction theorem, we develop a framework for
completion and reconstruction in the category of bands.
A \emph{band} is a pointed commutative monoid equipped with a
\emph{null set} $N_B$ of finite formal sums that encodes additive
relations and gives every element a unique additive inverse
(for the precise definition, see \autoref{dfn:band}). An
\emph{idyll} is a band in which every nonzero element is
multiplicatively invertible (cf.\ \autoref{dfn:idyll}); thus idylls
are the field-like bands. 
Commutative rings and commutative Krasner hyperrings embed fully
faithfully into the category of bands, while hyperfields give rise
to idylls
(see \cite[Section~1.2]{BJL25}).

Let $B$ be a band and let
$\mathcal U=(U_\lambda)_{\lambda\in\Lambda}$ be a decreasing family
of subgroups of the unit group $B^\times$, indexed by a directed set, with
$U_\mu\subseteq U_\lambda$ whenever $\mu\geq\lambda$
(cf.\ \autoref{dfn:filtered-band}). 
The group $U_\lambda$ acts on every element of $B$ by multiplication.
The quotient band $B/U_\lambda$ is the orbit monoid for this action 
(cf.\ \autoref{dfn:quotient-band}). 
We define the \emph{Krasner completion} of $(B,\mathcal U)$ by
\[
\widehat B_{\mathcal U}
=
\varprojlim_{\lambda\in\Lambda}B/U_\lambda.
\]
The filtration $\mathcal U$ gives $B$ a topology whose basic neighborhoods of
$a$ are the orbits $aU_\lambda$. For $r\geq1$, we define the \emph{$r$-th null locus} 
\[
\Null_r(B) = \Set{(a_1,\ldots,a_r)\in B^r |  a_1+ \cdots +a_r\in N_B}
\]
(cf.\ \autoref{dfn:null-locus}). 
Our main result gives a criterion
for the canonical morphism
$B\to\widehat B_{\mathcal U}$
to be an isomorphism.

\begin{thm}[{\autoref{thm:criterion}}]
\label{thm:intro-main}
The canonical morphism $B\to\widehat B_{\mathcal U}$ is an
isomorphism of bands if and only if
\begin{enumerate}[label=$(\alph*)$]
\item it is bijective as a morphism of underlying pointed commutative
monoids, and
\item $\Null_r(B)$ is closed in $B^r$ for every $r\geq1$.
\end{enumerate}
\end{thm}

%

As an application, we first
recover the theorem of Krasner--Linzi \autoref{eq:kras} directly from the general
reconstruction criterion. 
The dense embedding
$K\hookrightarrow\widehat K$ induces canonical isomorphisms
$K/U_v^\gamma\simeq\widehat K/U_{\widehat v}^\gamma$ of bands for
every $\gamma$ (\autoref{prop:finite-level-comparison}). For the
complete field $\widehat K$, the null loci are closed in the valuation
topology. 
Hence
\[
\widehat K
\simeq
\varprojlim_\gamma
\widehat K/U_{\widehat v}^\gamma
\simeq
\varprojlim_\gamma
K/U_v^\gamma.
\]

The valuation ring analogue is then obtained from the field case by
passing to the nonnegative-valuation subband. Let $R$ be a valuation
ring with fraction field $K$, value group $\Gamma$, and associated
valuation $v\colon K^\times\to\Gamma$, and put
\[
\widehat R=\set{x\in\widehat K| \widehat v(x)\geq0}.
\]
The dense embedding $R\hookrightarrow\widehat R$ induces 
isomorphisms
$R/U_v^\gamma\simeq\widehat R/U_{\widehat v}^\gamma$, while the
isomorphism \eqref{eq:kras} for $\widehat K$ identifies the inverse limit of
the valuation ring quotients with its nonnegative-valuation subband.
Thus
\begin{thm}[{\autoref{thm:valuation-ring-krasner}}]
Let $R$ be a valuation ring with value group $\Gamma$. Then there is
a canonical isomorphism of bands
\[
\widehat R
\xrightarrow{\simeq}
\varprojlim_{\gamma\in\Gamma_{\geq0}}R/U_v^\gamma.
\]
\end{thm}

%

A second branch of applications concerns subbands of fields.
For a subgroup $G\subseteq K^\times$ containing $-1$, the band
$B_G=G\cup\set{0}$ is a subband of the field-band $K$. The subband
reconstruction principle therefore yields a partial-field
reconstruction result from a multiplicative completeness hypothesis.
This produces reconstructible idylls which do not arise from
hyperfields (\autoref{cor:partial-field-reconstruction} and
\autoref{ex:nonhyperfield}).

Finally, 
let $R$
be a commutative ring and let $I\subseteq\operatorname{Jac}(R)$ be an
ideal, where $\operatorname{Jac}(R)$ denotes the Jacobson radical.
The unit band
$B^\times(R)=R^\times\cup\set{0}$ is a subband of $R$, and its Krasner
completion for $U_n=1+I^n$ is canonically identified with
$B^\times(\widehat R)$, where $\widehat R$ is the $I$-adic completion
(\autoref{thm:adic-comparison}). This shows that the same quotient-band
formalism also interacts naturally with ordinary adic completion.



\subsection*{Acknowledgements} 
The author was supported by JSPS KAKENHI Grant Number 24K06672.

\section{Bands and Krasner completion}
\label{sec:bands}

\subsection{Bands}
We follow the terminology of \cite[Section~1]{BJL25}.  Let $B$ be a pointed
commutative monoid. 
We write $B^+$ for the semiring of finite formal
sums of elements of $B$, with $0\in B$ identified with the empty sum.

\begin{dfn}\label{dfn:band}
A \emph{band} is a pointed commutative monoid $B$ together with an ideal
$N_B\subset B^+$, called the \emph{null set}, such that for every
$a\in B$ there is a unique $b\in B$ satisfying
$a+b\in N_B.$
The element $b$ is denoted by $-a$.
A finite formal sum $\sum_i a_i\in B^+$ is called \emph{null}
if $\sum_i a_i\in N_B$; such a sum 
$\sum_i a_i\in N_B$ is called a \emph{null relation}.

A \emph{morphism} of bands is a morphism $f\colon B\to C$ 
of pointed commutative monoids which preserves null relations, namely, 
$a_1+\cdots+a_r\in N_B \ \Rightarrow\ f(a_1)+\cdots+f(a_r)\in N_C.$
\end{dfn}

Suppose that
$f\colon B\to C$ 
is a morphism of bands whose underlying pointed commutative monoid map
is an isomorphism.  
The inverse map on the underlying pointed commutative monoids is a
morphism of bands if and only if $f$ also \emph{reflects null
relations}, namely 
\begin{equation}\label{eq:isom}
\mbox{if}\ f(a_1)+\cdots+f(a_r)\in N_C
\quad \mbox{then}\quad a_1+\cdots+a_r\in N_B
\end{equation}
for every $r\geq1$ and every $a_1,\ldots,a_r\in B$.  
We denote the category of bands by $\Bands$.
\begin{dfn}\label{dfn:idyll}
An \emph{idyll} is a band $B$ such that $0\neq1$ and every nonzero element is invertible.
\end{dfn}

\begin{dfn}\label{dfn:subband}
Let $C$ be a band. A \emph{subband} of $C$ is a pointed submonoid
$B\subseteq C$ containing $-1$, equipped with the induced null set
\[
N_B=B^+\cap N_C,
\]
where $B^+$ is identified with its natural subsemiring of $C^+$.
Equivalently, the inclusion $B\hookrightarrow C$ preserves and reflects
null relations. This terminology is consistent with the use of
subbands in \cite[Example~3.8]{BJL25}.
\end{dfn}

Every commutative ring $R$ defines a band with
\[
        N_R=\Set{\sum_i a_i\in R^+ |\sum_i a_i=0\text{ in }R}.
\]
A commutative hyperring $H$ in the sense of Krasner, whose addition is a multivalued operation $\boxplus$, defines a band with
\begin{equation}\label{eq:NH}
        N_H=\Set{\sum_i a_i\in H^+ | 0\in a_1\boxplus\cdots\boxplus a_r}.
\end{equation}
The construction $R\mapsto B_R$ defines a fully faithful functor
from the category of commutative rings to $\Bands$, and the construction
$H\mapsto B_H$ defines a fully faithful functor from the category of
commutative Krasner hyperrings to $\Bands$; see
\cite[Section~1.2]{BJL25}.

The category $\Bands$ of bands is known to be complete and cocomplete (\cite[Corollary~1.43]{BJL25}, see also \cite[Section~3.1.6]{Ham26}). 
We use the explicit description of inverse limits in $\Bands$. 
Indeed, limits in $\Bands$ are computed on the underlying pointed
monoids. Thus, if $B=\varprojlim_i B_i$ in $\Bands$, then
$B=\varprojlim_i B_i$ 
as \emph{monoids},
and a finite formal sum $\sum_\nu x_\nu\in B^+$ is null if and only if
\[
\sum_\nu \pi_i(x_\nu)\in N_{B_i}
\]
for every $i$, where $\pi_i\colon B\to B_i$ is the canonical
projection (see \cite[Section~1.11]{BJL25}).


\begin{dfn}
\label{dfn:quotient-band}
Let $B$ be a band and let $U\subset B^\times$ be a subgroup. The group $U$ acts on $B$ by multiplication. As a pointed commutative monoid, let $B/U$ be the orbit monoid, and write $[a]_U$ for the orbit of $a\in B$. Declare
$[a_1]_U+\cdots+[a_r]_U\in N_{B/U}$
if and only if there exist $u_1,\ldots,u_r\in U$ such that
$a_1u_1+\cdots+a_ru_r\in N_B.$
Zeros in a formal sum are ignored, and the empty sum is null.
\end{dfn}

The construction in \autoref{dfn:quotient-band} agrees with the quotient band 
$B\slash\!\!\slash I_U$ defined in \cite[Propositions~1.12 and~1.13]{BJL25} 
associated with the null ideal $I_U$ generated by the formal sums
$u-1$, $u\in U$.
Consequently, $B/U$ is a band and the canonical map
$q_U\colon B\to B/U$
is a morphism of bands.  If $B$ is an idyll, then $B/U$ is also an
idyll.

%

\begin{lem}
\label{lem:transition}
If $V\subset U\subset B^\times$ are subgroups, then
$[a]_V\mapsto[a]_U$
defines a morphism of bands $B/V\to B/U$.
\end{lem}

\begin{proof}
A null relation modulo $V$ is represented by a null relation in $B$
after multiplication of its terms by elements of $V$.  Since
$V\subset U$, the same representatives show that its image is null
modulo $U$.
\end{proof}

\subsection{Krasner completion}
Let $\Lambda$ be a directed set.

\begin{dfn}\label{dfn:filtered-band}
A \emph{filtered band} is a pair $(B,\mathcal U)$ consisting of a band $B$ and a family
$\mathcal U=(U_\lambda)_{\lambda\in\Lambda}$ 
of subgroups of $B^\times$ such that $U_\mu\subseteq U_\lambda$ whenever $\mu\geq\lambda$. 
The \emph{Krasner completion} of $(B,\mathcal U)$ is the inverse limit in $\Bands$ 
\[
\widehat B_{\mathcal U}:=\varprojlim_{\lambda\in\Lambda}B/U_\lambda.
\]
A filtered band is called a \emph{filtered idyll} if its underlying band is an idyll.
\end{dfn}

The filtration $\mathcal U$ gives $B$ the topology whose basic neighborhoods of $a\in B$ are $aU_\lambda$. Since $0U_\lambda=\set{0}$, the point $0$ is isolated.  We always give
$B^r$ the product topology.

\begin{dfn}\label{dfn:null-locus}
For $r\geq1$, the \emph{$r$-th null locus} of $B$ is
\[
        \Null_r(B)
        =\Set{(a_1,\ldots,a_r)\in B^r|a_1+\cdots+a_r\in N_B}.
\]
\end{dfn}

\begin{lem}
\label{lem:closure-null}
Let $a_1,\ldots,a_r\in B$.  The following are equivalent.
\begin{enumerate}
\item For every $\lambda\in\Lambda$, we have 
$[a_1]_\lambda+\cdots+[a_r]_\lambda \in N_{B/U_\lambda}.$
Here, we write $[a]_\lambda = [a]_{U_\lambda}$ in $B/U_\lambda$. 
\item The point $(a_1,\ldots,a_r)$ belongs to the closure $\overline{\Null_r(B)}$ of
$\Null_r(B)$ in $B^r$.
\end{enumerate}
\end{lem}

\begin{proof}
We may omit the coordinates for which $a_i=0$, since $0$ is isolated
and zeros do not affect null relations. 
Recall that $(b_1,\ldots,b_r)\in B^r$ is in the closure $\overline{\Null_r(B)}$ 
if and only if 
$(b_1U_{\lambda_1}\times \cdots \times b_r U_{\lambda_r}) \cap \Null_r(B)\neq \emptyset$ 
for any $\lambda_1,\ldots,\lambda_r\in \Lambda$. 

\smallskip
\noindent
(i) $\Rightarrow$ (ii).
First, we assume (i). 
Consider 
a basic neighborhood 
$a_1U_{\lambda_1}\times\cdots\times a_rU_{\lambda_r}$ 
of $(a_1,\ldots,a_r)$. 
Choose $\lambda$ such that  $U_\lambda \subset \cap_{i=1}^rU_{\lambda_i}$.  
From the assumption (i), 
$[a_1]_\lambda+\cdots+[a_r]_\lambda$ 
is in $N_{B/U_\lambda}$. 
By \autoref{dfn:quotient-band}, there are $u_i\in U_\lambda$ such
that 
$a_1u_1+\cdots+a_ru_r\in N_B.$
This implies 
\[
\emptyset \neq a_1U_\lambda \times \cdots \times a_rU_\lambda \cap \Null_r(B)\subset a_1U_{\lambda_1}\times \cdots \times a_rU_{\lambda_r} \cap \Null_r(B).
\]

\smallskip
\noindent
(ii) $\Rightarrow$ (i).
Conversely, for every $\lambda\in \Lambda$, 
$a_1U_\lambda \times \cdots \times a_rU_{\lambda} \cap \Null_r(B) \neq \emptyset.$
Therefore, 
there exist $u_i \in U_\lambda$ such that 
$a_1u_1+\cdots+a_ru_r\in N_B.$
By \autoref{dfn:quotient-band}, 
this implies $[a_1]_\lambda + \cdots + [a_r]_{\lambda}$ is in $N_{B/U_\lambda}$. 
%
%
\end{proof}

There is a canonical band morphism $\eta_B\colon
B\to\widehat B_{\mathcal U}$ given by
$a\mapsto([a]_\lambda)_\lambda$.
The preceding description of null relations at finite levels leads to
the following reconstruction criterion.

\begin{thm}
\label{thm:criterion}
Let $(B,\mathcal U)$ be a filtered band. The following are equivalent.
\begin{enumerate}
\item The canonical morphism $\eta_B:B\to\widehat B_{\mathcal U}$ is an isomorphism of bands.
\item $(\mathrm{a})$ The canonical map $B\to\widehat B_{\mathcal U}$ is an isomorphism of pointed commutative monoids, and \\
	$(\mathrm{b})$ $\Null_r(B)$ is closed in $B^r$ for every $r\geq1$.
\end{enumerate}
\end{thm}
\begin{proof}
%
%
(i) $\Rightarrow$ (ii).
Suppose that $\eta_B\colon B\to\widehat B_{\mathcal U}$ 
is an isomorphism of bands.  
Then its underlying pointed commutative monoid map is an isomorphism; the condition (a) holds.
We show that $\Null_r(B)$ is closed in $B^r$ for every $r\geq1$.
Fix $r\geq1$ and let
$(a_1,\ldots,a_r)\in\overline{\Null_r(B)}$.
By \autoref{lem:closure-null}, this is equivalent to
$[a_1]_\lambda+\cdots+[a_r]_\lambda \in N_{B/U_\lambda}$
for every $\lambda\in\Lambda$.
A finite formal sum in an inverse limit of bands is null if and only if
its image in every component is null. 
Hence
$\eta_B(a_1)+\cdots+\eta_B(a_r) \in N_{\widehat B_{\mathcal U}}.$
Since $\eta_B$ is an isomorphism of bands, it reflects null relations.
It follows that
$a_1+\cdots+a_r\in N_B,$
or equivalently, $(a_1,\ldots,a_r)\in\Null_r(B)$.
We have therefore proved that
$\overline{\Null_r(B)} =\Null_r(B)$.
Thus $\Null_r(B)$ is closed in $B^r$.

\smallskip
\noindent
(ii) $\Rightarrow$ (i). 
Suppose the conditions (a) and (b).
The underlying
pointed commutative monoid map is an isomorphism by (a). 
By \eqref{eq:isom}, 
it is enough to show that 
for $a_1,\ldots,a_r\in B$ with $\eta_B(a_1)+\cdots+\eta_B(a_r)$
in $N_{\widehat B_{\mathcal U}}$, 
we have $a_1 + \cdots + a_r \in N_B$.

Take $a_1,\ldots,a_r\in B$ with 
$\eta_B(a_1)+\cdots+\eta_B(a_r) \in N_{\widehat B_{\mathcal U}}$.
Since a finite formal sum in an inverse limit of bands is null if and
only if its image in every component is null, 
we obtain 
$[a_1]_\lambda+\cdots+[a_r]_\lambda \in N_{B/U_\lambda}$ 
for every $\lambda\in \Lambda$. 
Therefore, we have 
\[
(a_1,\ldots,a_r) \in \overline{\Null_r(B)}.  
\]
The condition (b) now implies
$\overline{\Null_r(B)}=\Null_r(B)$,
and hence
$a_1+\cdots+a_r\in N_B$.
Therefore $\eta_B$ reflects null relations, and hence
$\eta_B$ is an isomorphism of bands.%
\end{proof}


\begin{cor}
\label{cor:idyll-criterion}
Let $(B,\mathcal U)$ be a filtered idyll. Then the condition $(\mathrm{a})$ in $(\mathrm{ii})$ of \autoref{thm:criterion} is equivalent to the following condition: 
 
\noindent 
$(\mathrm{a}$\!'\,$)$ The canonical map
$\eta_B^\times\colon B^\times\to \varprojlim_\lambda B^\times/U_\lambda$ 
is bijective.
\end{cor}

\begin{proof}
For an idyll, $B=\set{0}\sqcup B^\times$ and $B/U_\lambda=\set{0}\sqcup B^\times/U_\lambda$. 
There is a natural identification
\[
\varprojlim_\lambda B/U_\lambda
=
\set{0}\sqcup
\varprojlim_\lambda B^\times/U_\lambda
\]
of pointed commutative monoids.
Under this natural identification, the canonical map
\[
\eta_B\colon
B\to \varprojlim_\lambda B/U_\lambda
\]
is given by $0\mapsto(0)_\lambda$ 
on the zero element and by the canonical map
$B^\times\to 
\varprojlim_\lambda B^\times/U_\lambda$,
on the nonzero part. Therefore, $\eta_B$ is an isomorphism of pointed
commutative monoids if and only if this canonical map on unit groups is
bijective. 
\end{proof}

The criterion above has the following useful form for bands whose null
relations are induced from an ambient band.

\begin{cor}
\label{cor:subband-reconstruction}
Let $C$ be a band equipped with a topology such that $\Null_r(C)$ is
closed in $C^r$ for every $r\geq1$, and let $B\subseteq C$ be a
subband. Let $\mathcal U=(U_\lambda)_{\lambda\in\Lambda}$ be a
filtration of $B$. Assume that the filtration topology on $B$ is finer
than the subspace topology inherited from $C$ and that the canonical
map
$B\to\varprojlim_\lambda B/U_\lambda$
is an isomorphism of pointed commutative monoids. Then the canonical
morphism
\[
B\xrightarrow{\simeq}\varprojlim_\lambda B/U_\lambda
\]
is an isomorphism of bands.
\end{cor}

\begin{proof}
For every $r\geq1$, the definition of a subband gives
$\Null_r(B)=B^r\cap\Null_r(C)$.
Thus $\Null_r(B)$ is closed for the subspace topology inherited from
$C$, and hence also for the finer filtration topology on $B$. The
assertion follows from \autoref{thm:criterion}.
\end{proof}

%


\section{Applications}
\label{sec:applications}

We now apply the reconstruction principles of
\autoref{sec:bands}. 

\subsection{Valued fields}

Let $(K,v)$ be a valued field with value group $\Gamma$ (for the definition, see \cite[Section~3]{Lin23}), put $v(0)=\infty$, and, for $\gamma\in\Gamma_{\geq0}$, set
$U_v^\gamma=\set{u\in K^\times | v(u-1)>\gamma}.$
Regard $K$ as a band and form the quotient band $K/U_v^\gamma$.

\begin{prop}\label{prop:field-quotient}
The quotient band $K/U_v^\gamma$ is canonically isomorphic to the band associated with Krasner's quotient hyperfield
\[
        H =K^\times/U_v^\gamma\cup\set{0}.
\]
\end{prop}
\begin{proof}
Put $U = U_v^\gamma$ and 
let
$\varphi\colon K/U \to  H$
be the identity on the underlying pointed set, namely
$\varphi(0)=0$,
$\varphi([a]_{U})=[a]_{U}$
($a\in K^\times$).
The map $\varphi$ is an isomorphism of pointed commutative monoids.

It remains to compare the null relations: 
\[
[a_1]_{U} + \cdots + [a_r]_{U} \in N_{K/U}\ \mbox{if and only if}\ 
\varphi([a_1]_{U_v^\gamma})+\cdots+
\varphi([a_r]_{U_v^\gamma}) \in N_{H}.
\]
Here, the null set $N_{H}$ of the hyperfield $H$ is defined by 
\[
N_{H} = \Set{\sum_{i=1}^r [a_i]_{U} \in H^{+} | 0\in [a_1]_U\boxplus \cdots\boxplus [a_r]_{U} }
\]
as we recalled in \eqref{eq:NH}.

 Let
$a_1,\ldots,a_r\in K$, where zero terms may be omitted.  
Suppose that 
$[a_1]_U+ \cdots + [a_r]_U\in N_{K/U}$. 
By the definition of the quotient band \autoref{dfn:quotient-band}, 
there exist $u_1,\ldots,u_r\in U$ such that 
$a_1u_1+ \cdots + a_ru_r \in N_K.$
Since the null set of the band associated with the field $K$ is
\[
N_K
=
\Set{
\sum_{i=1}^r x_i\in K^+
|\sum_{i=1}^r x_i=0\text{ in }K},
\]
we have $a_1u_1+ \cdots + a_ru_r =0$. 
This implies 
$0\in [a_1]_{U_v^\gamma}\boxplus\cdots\boxplus [a_r]_{U_v^\gamma}.$
Therefore, 
$\varphi([a_1]_{U})+\cdots+ \varphi([a_r]_{U}) \in N_{H}$
holds. 

Conversely, 
we assume 
$\varphi([a_1]_{U})+\cdots+ \varphi([a_r]_{U}) \in N_{H}.$
This gives 
$0\in [a_1]_{U_v^\gamma}\boxplus\cdots\boxplus [a_r]_{U_v^\gamma}.$
By the definition of Krasner's quotient hyperaddition, this is
equivalent to the existence of
$u_1,\ldots,u_r\in U_v^\gamma$ 
such that
$a_1u_1+\cdots+a_ru_r=0$
in $K$, namely, 
$a_1u_1+\cdots+a_ru_r\in N_K.$
By the definition of the quotient band $K/U$, this implies
$[a_1]_{U}+\cdots+[a_r]_{U} \in N_{K/U}.$
Hence $\varphi$ preserves and reflects null relations, and therefore
is an isomorphism of bands.
\end{proof}


\begin{prop}\label{prop:finite-level-comparison}
Let $(\widehat{K},\widehat{v})$ be the completion of $(K,v)$. For every $\gamma\in\Gamma_{\geq0}$, the dense embedding $K\hookrightarrow \widehat{K}$ induces a canonical isomorphism of bands
\[
        K/U_v^\gamma\xrightarrow{\simeq}
        \widehat{K}/U_{\widehat{v}}^\gamma.
\]
\end{prop}

\begin{proof}
On nonzero elements, the map is
$K^\times/U_v^\gamma\to (\widehat{K})^\times/U_{\widehat{v}}^\gamma.$
It is injective because $K^\times\cap U_{\widehat{v}}^\gamma=U_v^\gamma$. To prove surjectivity, let $x\in(\widehat{K})^\times$. 
Since $K$ is dense in
$\widehat{K}$, the open neighborhood
$\set{y\in \widehat{K}|\widehat{v}(y-x)>\widehat{v}(x)+\gamma}$
of $x$ contains an element $a\in K$.  Since $\gamma\geq0$, we have
$\widehat{v}(a-x)>\widehat{v}(x)$,
and hence $\widehat{v}(a)=\widehat{v}(x)$; in particular, $a\neq0$.  Moreover,
\[
\widehat{v}(a/x-1)
=
\widehat{v}((a-x)/x)
=
\widehat{v}(a-x)-\widehat{v}(x)
>
\gamma.
\]
Thus $a/x\in U_{\widehat{v}}^\gamma$, and therefore
$[a]_{U_{\widehat{v}}^\gamma}=[x]_{U_{\widehat{v}}^\gamma}$.
Hence every nonzero class in $\widehat{K}/U_{\widehat{v}}^\gamma$ has a representative
in $K^\times$.

It remains to compare null relations. 
By \autoref{prop:field-quotient}, the source and target are the bands
associated with the Krasner quotient hyperfields $K/U_v^\gamma$ and
$\widehat{K}/U_{\widehat{v}}^\gamma$, respectively. 
The canonical bijection
$K/U_v^\gamma\to \widehat{K}/U_{\widehat{v}}^\gamma$ is an isomorphism of hyperfields by \cite[Section~4, Lemma~2]{Lin23}.
Consequently, it induces an isomorphism of the associated bands.
%
\end{proof}

\begin{thm}[{Krasner--Linzi \cite[Theorem~2]{Lin23}}]\label{thm:krasner}
Let $(\widehat{K},\widehat{v})$ be the completion of $(K,v)$. There is a canonical isomorphism of bands
\[
        \widehat{K}\xrightarrow{\simeq}
        \varprojlim_{\gamma\in\Gamma_{\geq0}}K/U_v^\gamma.
\]
\end{thm}
\begin{proof}
By \autoref{prop:finite-level-comparison}, the canonical isomorphisms
$K/U_v^\gamma
\xrightarrow{\simeq}
\widehat{K}/U_{\widehat{v}}^\gamma$ 
($\gamma\in\Gamma_{\geq0}$) 
are compatible with the transition morphisms.  Hence they induce an
isomorphism of bands
\[
\varprojlim_{\gamma\in\Gamma_{\geq0}}K/U_v^\gamma
\xrightarrow{\simeq}
\varprojlim_{\gamma\in\Gamma_{\geq0}}
\widehat{K}/U_{\widehat{v}}^\gamma.
\]
It is therefore enough to prove the assertion for a complete valued
field.  Replacing $(K,v)$ by $(\widehat{K},\widehat{v})$, we assume from now on that
$(K,v)$ is complete.

Endow $K$ with its valuation topology. For every $r\geq1$, the set
\[
\Null_r(K)
=
\Set{
(x_1,\ldots,x_r)\in K^r| x_1+ \cdots + x_r=0}
\]
is closed, being the inverse image of $\set{0}$ under the continuous
sum map $K^r\to K$. The topology on $K^\times$ defined by the
subgroups $U_v^\gamma$ agrees with the topology induced by the
valuation topology, since for $a\in K^\times$,
\[
aU_v^\gamma
=
\set{x\in K^\times| v(x-a)>v(a)+\gamma}.
\]
Consequently, the filtration topology on $K$ is finer than the
valuation topology; it additionally makes $0$ isolated. Hence every
$\Null_r(K)$ is closed in the filtration topology. 
By \autoref{thm:criterion} and \autoref{cor:idyll-criterion}, it is enough
to show that the canonical map
\[
\eta_K^\times\colon
K^\times\to
\varprojlim_{\gamma\in\Gamma_{\geq0}}
K^\times/U_v^\gamma
\]
is bijective.
We first prove that it is injective.  Suppose that $a,b\in K^\times$
have the same image.  Then
$a^{-1}b\in U_v^\gamma$
for every $\gamma\in\Gamma_{\geq0}$.  Since
$\bigcap_{\gamma\in\Gamma_{\geq0}}U_v^\gamma=\set{1}$,
we obtain $a=b$.
We next prove surjectivity.  Let
$(\xi_\gamma)_\gamma
\in
\varprojlim_{\gamma\in\Gamma_{\geq0}}
K^\times/U_v^\gamma$.
For every $\gamma$, choose a representative
$a_\gamma\in K^\times$ such that
$\xi_\gamma=a_\gamma U_v^\gamma$.
If $\delta\geq\gamma$, compatibility gives
$a_\delta a_\gamma^{-1}\in U_v^\gamma$.
In particular,
$v(a_\delta)=v(a_\gamma)$.
Indeed, every element of $U_v^\gamma$ has valuation zero.  Since the
indexing set is directed, there is therefore an element
$\alpha\in\Gamma$ such that
$v(a_\gamma)=\alpha$
for every $\gamma$.

If $\Gamma_{\geq0}=\set{0}$, then $U_v^0=\set{1}$ and surjectivity is
immediate. Assume henceforth that $\Gamma$ is nontrivial.
Let $\kappa$ be the cofinality of $\Gamma$, that is, the least
ordinal for which there exists a cofinal sequence in $\Gamma$.
Choose an increasing sequence
$(\gamma_\nu)_{\nu<\kappa}$ 
of elements of $\Gamma_{\geq0}$ which is cofinal in $\Gamma$: for
every $\delta\in\Gamma$, there exists $\nu<\kappa$ such that
$\delta<\gamma_\nu$.
Following \cite[Section~3]{Lin23}, 
a sequence
$(x_\nu)_{\nu<\kappa}$ in $K$ is Cauchy if, for every
$\beta\in\Gamma$, there exists $\nu_0<\kappa$ such that
$v(x_\nu-x_\mu)>\beta$ 
whenever $\nu,\mu\geq\nu_0$, and $K$ is complete if every such
Cauchy sequence converges in $K$.

We claim that $(a_{\gamma_\nu})_{\nu<\kappa}$ is a Cauchy sequence. 
Let $\beta\in\Gamma$. 
Since $(\gamma_\nu)_{\nu<\kappa}$ is cofinal in $\Gamma$, choose 
$\nu_0<\kappa$ such that 
$\beta - \alpha< \gamma_{\nu_0}$. 
Then  $\alpha+\gamma_{\nu_0}>\beta$. 
If
$\mu,\nu\geq\nu_0$, compatibility implies that
$a_{\gamma_\mu}$ and $a_{\gamma_\nu}$ have the same class modulo
$U_v^{\gamma_{\nu_0}}$. Consequently,
\[
v(a_{\gamma_\mu}-a_{\gamma_\nu}) 
= v(a_{\gamma_\mu}) + v(a_{\gamma_\nu}a_{\gamma_{\mu}}^{-1}-1)
>
\alpha+\gamma_{\nu_0}
>\beta.
\]
Thus $(a_{\gamma_\nu})_{\nu<\kappa}$ is Cauchy. Since $K$ is
complete, it converges to some $a\in K$.

The limit $a$ is nonzero and satisfies $v(a)=\alpha$. Indeed, for all
sufficiently large $\nu$,
$v(a_{\gamma_\nu}-a)>\alpha$,
and the ultrametric inequality then gives
$v(a)=v(a_{\gamma_\nu})=\alpha$.

We show that $a$ represents the original compatible family
$(\xi_\gamma)_\gamma$. Fix $\gamma\in\Gamma_{\geq0}$. By cofinality,
choose $\nu_0<\kappa$ such that $\gamma_{\nu_0}\geq\gamma$, and then
choose $\nu\geq\nu_0$ sufficiently large that
$v(a_{\gamma_\nu}-a)>\alpha+\gamma$. Compatibility gives
$a_{\gamma_\nu}a_\gamma^{-1}\in U_v^\gamma$, and hence
$v(a_{\gamma_\nu}-a_\gamma)>\alpha+\gamma$. It follows that
$v(a-a_\gamma)>\alpha+\gamma$. Since
$v(a_\gamma)=\alpha$, we obtain
$v(aa_\gamma^{-1}-1)>\gamma$. Thus
$aa_\gamma^{-1}\in U_v^\gamma$ and hence
$aU_v^\gamma=\xi_\gamma$. This holds for every $\gamma$, so
$\eta_K^\times$ is surjective.
\end{proof}

\subsection{Valuation rings}

We now pass from the field case to its nonnegative-valuation subband.
We recall that a valuation ring $R$
with fraction field $K$ determines a totally ordered abelian group
$\Gamma=K^\times/R^\times$
and an associated valuation
$v\colon K^\times\to\Gamma.$
With the convention $v(0)=\infty$, one has
\[
R=\set{x\in K| v(x)\geq0},
\qquad
R^\times=\set{x\in K^\times | v(x)=0}.
\]
For these standard facts about valuation rings and their associated
valuations, see
\cite[Section~10.50, especially Tags~00IE and~00IG]{Stacks};
see also \cite[Chapter~2]{EP05}.

Let $(\widehat K,\widehat v)$ be the completion of the valued field
$(K,v)$, in the sense used in \cite[Section~3]{Lin23}, and put
\[
\widehat R
=
\set{x\in\widehat K | \widehat v(x)\geq0}.
\]
The valuation $\widehat v$ has the same value group $\Gamma$.  
Indeed,
take $x\in\widehat K^\times$. 
Since $K$ is dense in $\widehat K$, the open neighborhood
$\set{y\in\widehat K | \widehat v(y-x)>\widehat v(x)}$ of $x$
contains an element $a\in K$.
Therefore, we have 
$\widehat v(a-x)>\widehat v(x)$, 
and hence
$v(a)=\widehat v(x)$.

The ring $\widehat R$ is canonically the completion of $R$ for the
valuation topology.  To see density, let
$x\in\widehat R\smallsetminus\set{0}$ and
$\gamma\in\Gamma_{\geq0}$.  Since $K$ is dense in $\widehat K$, there
exists $a\in K$ such that
$\widehat v(a-x)>\widehat v(x)+\gamma.$
Then
$v(a)=\widehat v(x)\geq0,$
so $a\in R$.  Moreover, $\widehat R$ is closed in $\widehat K$.  In
fact, if a net in $\widehat R$ converges to an element
$x\in\widehat K$ with $\widehat v(x)<0$, then its sufficiently late
terms have valuation $\widehat v(x)<0$, which is impossible.
Therefore $\widehat R$ is complete.

We regard $R$ and $\widehat R$ as the bands associated with their
underlying rings.  For every $\gamma\in\Gamma_{\geq0}$, put
\begin{align*}	
U_v^\gamma
=\set{u\in R^\times | v(u-1)>\gamma }\quad \mbox{and}\quad 
U_{\widehat v}^\gamma
=\set{u\in\widehat R^\times|\widehat v(u-1)>\gamma}.
\end{align*}
We first give a valuation-theoretic description of the null relations
in the corresponding quotient bands.

\begin{lem}
\label{lem:valuation-ring-null}
Let $a_1,\ldots,a_r\in R$ be nonzero, and put
$\alpha = \min_{1\leq i\leq r}v(a_i).$
For every $\gamma\in\Gamma_{\geq0}$, the following are equivalent:
\begin{enumerate}
\item
$[a_1]_{U_v^\gamma} +\cdots+ [a_r]_{U_v^\gamma} \in N_{R/U_v^\gamma}.$
\item
$v(a_1+\cdots+a_r)>\alpha+\gamma.$
\end{enumerate}
The analogous equivalence holds in $\widehat R$, with $v$ replaced
by $\widehat v$.
\end{lem}

\begin{proof}
Suppose first that
$[a_1]_{U_v^\gamma} +\cdots+ [a_r]_{U_v^\gamma} \in N_{R/U_v^\gamma}.$
By \autoref{dfn:quotient-band}, there exist
$u_1,\ldots,u_r\in U_v^\gamma$ such that
$a_1u_1+\cdots+a_ru_r=0.$
Consequently,
\[
a_1+\cdots+a_r
=
-\sum_{i=1}^r a_i(u_i-1).
\]
For every $i$, we have
$v\bigl(a_i(u_i-1)\bigr) = v(a_i)+v(u_i-1) > v(a_i)+\gamma \geq \alpha+\gamma.$
The ultrametric inequality therefore gives
$v(a_1+\cdots+a_r)>\alpha+\gamma.$

Conversely, put
$s=a_1+\cdots+a_r$
and suppose that
$v(s)>\alpha+\gamma.$
If $s=0$, we may take $u_i=1$ for every $i$.  Suppose that $s\neq0$.
Choose an index $j$ such that
$v(a_j)=\alpha,$
and put
\[
u_j=1-a_j^{-1}s,
\qquad
u_i=1
\quad(i\neq j).
\]
Then
$v(a_j^{-1}s) = v(s)-v(a_j) > \gamma \geq0.$
Thus $a_j^{-1}s\in R$.  Since
$v(a_j^{-1}s)>0,$
we have
$v(u_j) = v(1-a_j^{-1}s) = 0,$
and hence $u_j\in R^\times$.  Moreover,
$v(u_j-1) = v(a_j^{-1}s) > \gamma,$
so $u_j\in U_v^\gamma$.  Finally,
\[
a_1u_1+\cdots+a_ru_r = s-a_ja_j^{-1}s = 0.
\]
Therefore
$[a_1]_{U_v^\gamma} +\cdots+ [a_r]_{U_v^\gamma} \in N_{R/U_v^\gamma}.$
The same proof applies to the valuation ring $\widehat R$.
\end{proof}

\begin{prop}
\label{prop:valuation-ring-finite-level-comparison}
For every $\gamma\in\Gamma_{\geq0}$, the dense embedding
$R\hookrightarrow\widehat R$ induces a canonical isomorphism of bands
\[
R/U_v^\gamma
\xrightarrow{\simeq}
\widehat R/U_{\widehat v}^\gamma.
\]
\end{prop}

\begin{proof}
The embedding $R\hookrightarrow\widehat R$ maps $U_v^\gamma$ into
$U_{\widehat v}^\gamma$ and therefore induces a morphism of bands
$\varphi_\gamma\colon R/U_v^\gamma \to \widehat R/U_{\widehat v}^\gamma.$

We first show that $\varphi_\gamma$ is an isomorphism of underlying
pointed commutative monoids.  Suppose that
$a,b\in R\smallsetminus\set{0}$ have the same image in
$\widehat R/U_{\widehat v}^\gamma$.  Then
$ab^{-1}\in U_{\widehat v}^\gamma$.
Since $a,b\in K^\times$ and $\widehat v$ restricts to $v$ on $K$, we
obtain $v(ab^{-1}) = 0$ so that $ab^{-1}\in R^\times$. 
Moreover,  
by $v(ab^{-1}-1)>\gamma$, 
$ab^{-1}\in U_v^\gamma$. 
Hence, 
$a$ and $b$ have the same class in $R/U_v^\gamma$.  Thus
$\varphi_\gamma$ is injective.

To prove surjectivity, let
$x\in\widehat R\smallsetminus\set{0}$.  Since $K$ is dense in
$\widehat K$, there exists $a\in K$ such that
$\widehat v(a-x)>\widehat v(x)+\gamma$.
Since
$\widehat v(a-x)>\widehat v(x)$,
the ultrametric inequality gives
$v(a)=\widehat v(x)\geq0$.
Thus
$a\in R\smallsetminus\set{0}$.
Furthermore,
$\widehat v(a/x-1) = \widehat v(a-x)-\widehat v(x) > \gamma$.
Hence
$a/x\in U_{\widehat v}^\gamma$,
so $a$ and $x$ determine the same class in
$\widehat R/U_{\widehat v}^\gamma$.  Therefore
$\varphi_\gamma$ is surjective.

It remains to compare null relations.  Let
$a_1,\ldots,a_r\in R$ be nonzero, and put
$\alpha=\min_i v(a_i).$
By \autoref{lem:valuation-ring-null},
$[a_1]_{U_v^\gamma} +\cdots+ [a_r]_{U_v^\gamma} \in N_{R/U_v^\gamma}$
if and only if
$v(a_1+\cdots+a_r)>\alpha+\gamma.$
The corresponding criterion in $\widehat R$ shows that
$[a_1]_{U_{\widehat v}^\gamma}
+\cdots+
[a_r]_{U_{\widehat v}^\gamma}
\in N_{\widehat R/U_{\widehat v}^\gamma}$
if and only if
$\widehat v(a_1+\cdots+a_r)>\alpha+\gamma.$
Since $\widehat v$ restricts to $v$ on $K$, these two conditions are
equivalent.  Zero terms may be omitted, since zeros do not affect
null relations.  Thus $\varphi_\gamma$ preserves and reflects null
relations and is therefore an isomorphism of bands.
\end{proof}

\begin{thm}
\label{thm:valuation-ring-krasner}
Let $R$ be a valuation ring with value group $\Gamma$, and let
$\widehat R$ be its completion for the valuation topology.  For
$\gamma\in\Gamma_{\geq0}$, put
$U_v^\gamma
=\set{u\in R^\times|v(u-1)>\gamma}$.
There is a canonical isomorphism of bands
\[
\widehat R
\xrightarrow{\simeq}
\varprojlim_{\gamma\in\Gamma_{\geq0}}
R/U_v^\gamma.
\]
\end{thm}

\begin{proof}
By \autoref{prop:valuation-ring-finite-level-comparison}, the compatible
finite-level isomorphisms
$R/U_v^\gamma
\xrightarrow{\simeq}
\widehat R/U_{\widehat v}^\gamma$ 
induce an isomorphism
\[
\varprojlim_{\gamma\in\Gamma_{\geq0}}R/U_v^\gamma
\xrightarrow{\simeq}
\varprojlim_{\gamma\in\Gamma_{\geq0}}
\widehat R/U_{\widehat v}^\gamma.
\]
It is therefore enough to reconstruct $\widehat R$ from the quotients
on the right.

For every $\gamma$, the inclusion
$\widehat R\subseteq\widehat K$ induces an embedding of bands
$\widehat R/U_{\widehat v}^\gamma
\hookrightarrow
\widehat K/U_{\widehat v}^\gamma$.
Indeed, equality of two nonzero classes in the target is already
detected by a quotient in $\widehat R$, since
$U_{\widehat v}^\gamma\subseteq\widehat R^\times$, and null relations
are reflected because all representatives and all multipliers lie in
$\widehat R$. These embeddings are compatible with the transition
maps and hence induce an embedding
\[
\varprojlim_\gamma
\widehat R/U_{\widehat v}^\gamma
\hookrightarrow
\varprojlim_\gamma
\widehat K/U_{\widehat v}^\gamma.
\]

By \autoref{thm:krasner}, the target is canonically isomorphic to
$\widehat K$. Under this identification, the image of the above
embedding is precisely the nonnegative-valuation subband
$\widehat R$. Indeed, if $x\in\widehat R$, then every component
$[x]_{U_{\widehat v}^\gamma}$ lies in
$\widehat R/U_{\widehat v}^\gamma$. Conversely, if a nonzero
$x\in\widehat K$ corresponds to an element of the inverse limit on the
left, then for any $\gamma$ its component is represented by some
$a_\gamma\in\widehat R\smallsetminus\set{0}$. Since every element of
$U_{\widehat v}^\gamma$ has valuation zero,
$\widehat v(x)=\widehat v(a_\gamma)\geq0$,
and hence $x\in\widehat R$. The zero family corresponds to
$0\in\widehat R$.

Thus the Krasner isomorphism restricts to a canonical isomorphism of
bands
\[
\widehat R
\xrightarrow{\simeq}
\varprojlim_{\gamma\in\Gamma_{\geq0}}
\widehat R/U_{\widehat v}^\gamma.
\]
Combining this with the finite-level comparison above gives the
asserted isomorphism.
\end{proof}


\begin{ex}
For a field $k$, 
let $R=k[t]_{(t)}$ with its associated $t$-adic valuation.  Its value group is
$\mathbb Z$, and its completion for the valuation topology is
$\widehat R=k[[t]].$
For $n\in\mathbb Z_{\geq0}$, we have
$U_v^n = 1+t^{n+1}R.$
Therefore \autoref{thm:valuation-ring-krasner} gives a canonical
isomorphism of bands
\[
k[[t]]
\xrightarrow{\simeq}
\varprojlim_{n \geq 0}
k[t]_{(t)}/(1+t^{n+1}k[t]_{(t)}).
\]
The band $k[[t]]$ is not an idyll, since $t$ is a nonzero nonunit.

In \cite[Example~4]{Lin23}, Linzi observes that the elements
of nonnegative valuation in the quotient hyperfields associated with
$k(\!(t)\!)$ correspond to truncated power series, and that the limit
construction recovers the full information of $k[\![t]\!]$.
\end{ex}

\subsection{Partial-field bands}

The following construction supplies a broad class of reconstructible
idylls which need not arise from hyperfields.

\begin{dfn}
Let $K$ be a field and let $G\subseteq K^\times$ be a subgroup
containing $-1$.  Define
$B_G=G\cup\set{0}$
with multiplication inherited from $K$ and with null set
\[
N_{B_G}
=
\Set{
\sum_{i=1}^r a_i\in B_G^+|\sum_{i=1}^r a_i=0\text{ in }K}.
\]
\end{dfn}
The object $B_G$ is an idyll.
%
In fact, the inherited null set is an ideal.  Since $-1\in G$, the additive
inverse of $a\in G$ is $-a\in G$, and it is unique because it is unique
in $K$.  Every nonzero element belongs to the group $G$.

\begin{cor}
\label{cor:partial-field-reconstruction}
Let $K$ be a Hausdorff topological field, let
$G\subseteq K^\times$ be a subgroup containing $-1$, and let
$\mathcal U=(U_\lambda)_{\lambda\in\Lambda}$ be a filtration of $G$.
Assume that the canonical map
$G\to\varprojlim_\lambda G/U_\lambda$ 
is bijective and that the filtration topology on $G$ is at least as
fine as the subspace topology inherited from $K$. Then there is a
canonical isomorphism of bands
\[
B_G\xrightarrow{\simeq}\varprojlim_\lambda B_G/U_\lambda.
\]
\end{cor}

\begin{proof}
Regard $K$ as the band associated with the field $K$. By construction,
$B_G\subseteq K$ is a subband. Since $K$ is Hausdorff,
\[
\Null_r(K)
=
\set{(x_1,\ldots,x_r)\in K^r | x_1+\cdots+x_r=0}
\]
is closed for every $r\geq1$.
By \autoref{cor:idyll-criterion}, the
assumed bijectivity of
\[
G\to\varprojlim_\lambda G/U_\lambda
\]
is equivalent to reconstruction of the underlying pointed commutative
monoid of $B_G$. The filtration topology on $B_G$ is finer than the
subspace topology inherited from $K$: this holds on $G$ by assumption,
and the filtration topology makes $0$ isolated. Hence all hypotheses of
\autoref{cor:subband-reconstruction} are satisfied.
\end{proof}

%
\begin{ex}
\label{ex:nonhyperfield}
Let $p>3$ and put
$G = p^{\mathbb Z}\cdot\set{\pm1}\cdot(1+p\mathbb Z_p) \subseteq\mathbb Q_p^\times.$
Let
\[
B_G=G\cup\set{0},
\qquad
U_n=1+p^n\mathbb Z_p
\quad(n\geq1),
\]
and put $\mathcal U=(U_n)_{n\geq1}$.

Endow $G$ with the subspace topology inherited from
$\mathbb Q_p^\times$.  The decomposition
\[
G
\simeq
p^{\mathbb Z}\times\set{\pm1}\times(1+p\mathbb Z_p)
\]
is an isomorphism of topological groups, where the first two factors
are discrete.  Moreover,
\[
1+p\mathbb Z_p
\xrightarrow{\simeq}
\varprojlim_n
(1+p\mathbb Z_p)/(1+p^n\mathbb Z_p).
\]
It follows that $(U_n)_{n\geq1}$ is a neighborhood basis of $1$ in
$G$, that its filtration topology agrees with the subspace topology,
and that the canonical map
\[
G\to\varprojlim_n G/U_n
\]
is bijective.  Therefore \autoref{cor:partial-field-reconstruction} gives a
canonical isomorphism
\[
B_G\xrightarrow{\simeq}\varprojlim_n B_G/U_n.
\]

On the other hand, $B_G$ does not arise from a hyperfield.  Indeed,
suppose that $B_G$ were the band associated with a hyperfield.  Since
the hyper-sum $1\boxplus1$ is nonempty, choose
$c\in1\boxplus1$.  Then
$1+1-c\in N_{B_G},$
and hence $c=2$ in $\mathbb Q_p$.  However, $2\notin G$.  Indeed, if
$2\in G$, then $v_p(2)=0$ would imply
$
2\in\set{\pm1}\cdot(1+p\mathbb Z_p)$,
and therefore
$2\equiv\pm1\pmod p$,
which is impossible for $p>3$.  Thus $B_G$ is an idyll which is reconstructed from its quotient bands but
does not arise from a hyperfield.
\end{ex}
%
%
%

\subsection{Unit bands}

Let $R$ be a commutative ring and let
$I\subseteq\operatorname{Jac}(R)$ be an ideal, 
where
$\operatorname{Jac}(R)$ denotes the Jacobson radical. 
Put $\widehat R=\varprojlim_n R/I^n$.
Associate with $R$ the \emph{unit band}
\begin{equation}
    \label{def:uband}
B^\times(R)=R^\times\cup\set{0},
\end{equation}
with multiplication inherited from $R$ and with null relations given
by the finite sums which vanish in $R$. Thus $B^\times(R)$ is a
subband of the band associated with $R$ in the sense of
\autoref{dfn:subband}. Put
$U_n=1+I^n\subseteq R^\times$,
$\mathcal U=(U_n)_{n\geq1}$.
The inclusion $I\subseteq\operatorname{Jac}(R)$ implies that
$1+I^n\subseteq R^\times$ and that every unit of $R/I^n$ lifts to a
unit of $R$.

\begin{prop}
\label{prop:adic-finite-level}
For every $n\geq1$, reduction modulo $I^n$ induces a canonical
isomorphism of bands
\[
B^\times(R)/U_n\xrightarrow{\simeq}B^\times(R/I^n).
\]
\end{prop}

\begin{proof}
Since $I^n\subseteq\operatorname{Jac}(R)$, reduction modulo $I^n$
induces an isomorphism
\[
R^\times/(1+I^n)
\xrightarrow{\simeq}
(R/I^n)^\times.
\]
It remains to compare null relations.

Let $a_1,\ldots,a_r\in R^\times$. Suppose first that the sum of
the classes $[a_1]_{U_n},\ldots,[a_r]_{U_n}$ is null in
$B^\times(R)/U_n$.
By \autoref{dfn:quotient-band}, there exist
$u_1,\ldots,u_r\in1+I^n$ such that
$a_1u_1+\cdots+a_ru_r=0$.  It follows that
\[
a_1+\cdots+a_r
=
-\sum_{i=1}^r a_i(u_i-1)
\in I^n.
\]
Thus the reductions of the $a_i$ have null sum in
$B^\times(R/I^n)$.

Conversely, suppose that
$a_1+\cdots+a_r\in I^n$.  Set
$u_1 = 1-a_1^{-1}(a_1+\cdots+a_r) \in1+I^n$
and $u_i=1$ for $i>1$.  Then
$a_1u_1+\cdots+a_ru_r=0$.  Therefore
$[a_1]_{U_n}+\cdots+[a_r]_{U_n} \in N_{B^\times(R)/U_n}.$
Thus reduction preserves and reflects null relations.
\end{proof}

\begin{thm}
\label{thm:adic-comparison}
There is a canonical isomorphism of bands
\[
B^\times(\widehat R)
\xrightarrow{\simeq}
\widehat{B^\times(R)}_{\mathcal U}
=
\varprojlim_{n\geq1}B^\times(R)/U_n.
\]
\end{thm}

\begin{proof}
By \autoref{prop:adic-finite-level}, it is enough to identify
$B^\times(\widehat R)$ with $\varprojlim_{n\geq1}B^\times(R/I^n)$.
Since
$\widehat R=\varprojlim_{n\geq1}R/I^n$,
the canonical map
\[
\widehat R^\times
\to
\varprojlim_{n\geq1}(R/I^n)^\times
\]
is bijective.  Indeed, a unit determines a compatible family of units,
and conversely the inverses of a compatible family of units are again
compatible and determine its inverse in $\widehat R$.  Hence the
canonical morphism
\[
B^\times(\widehat R)
\to
\varprojlim_{n\geq1}B^\times(R/I^n)
\]
is an isomorphism of underlying pointed commutative monoids.

It also preserves and reflects null relations: for
$a_1,\ldots,a_r\in B^\times(\widehat R)$,
$a_1+\cdots+a_r=0\text{ in }\widehat R$
if and only if its image is zero in $R/I^n$ for every $n$.  By the
description of null relations in inverse limits of bands, the above
canonical morphism is therefore an isomorphism of bands.  Combining
this with the compatible isomorphisms of
\autoref{prop:adic-finite-level} proves the assertion.
\end{proof}

\begin{ex}
Let $d\geq1$ and put
$R=k[x_1,\ldots,x_d]_{(x_1,\ldots,x_d)}$,
$I=(x_1,\ldots,x_d)$.
Then
$\widehat R=k[[x_1,\ldots,x_d]]$,
and \autoref{thm:adic-comparison} gives a canonical isomorphism
\[
B^\times(k[[x_1,\ldots,x_d]])
\xrightarrow{\simeq}
\widehat{B^\times(R)}_{\mathcal U}
=
\varprojlim_{n\geq1} B^\times(R)/(1+I^n).
\]

Likewise, let
$R=k[x,y]_{(x,y)}/(xy)$, $I=(x,y)$.
Then
$\widehat R=k[[x,y]]/(xy)$,
and hence there is a canonical isomorphism
\[
B^\times(k[[x,y]]/(xy))
\xrightarrow{\simeq}
\widehat{B^\times(R)}_{\mathcal U}
=
\varprojlim_{n\geq1} B^\times(R)/(1+I^n).
\]

The bands on the left do not arise from hyperfields.  Indeed, let
$t=x_1$ in the first ring and let $t=x$ in the second ring.  In either
case, $t$ is a nonzero nonunit, whereas $1$ and $-1+t$ are units.
Suppose that the corresponding unit band $B^\times(\widehat R)$ arose from a
hyperfield.  Since
$1\boxplus(-1+t)$
is nonempty, choose $c\in1\boxplus(-1+t)$.  Then
$1+(-1+t)-c\in N_{B^\times(\widehat R)},$
so $c=t$ in $\widehat R$.  This is impossible because
$t\notin B^\times(\widehat R) = \widehat R^\times\cup\set{0}.$
Thus these Krasner completions do not arise from hyperfields.
\end{ex}

\end{document}